\documentclass[11pt]{amsart}

\usepackage[english]{babel}
\usepackage{amsfonts}
\usepackage{amsthm}
\usepackage{amssymb}
\usepackage{array}
\usepackage{booktabs}
\usepackage{caption}
\usepackage{longtable}
\newcolumntype{L}{>{$\displaystyle}l<{$}}
\newcolumntype{C}{>{$\displaystyle}c<{$}}
\newcolumntype{R}{>{$\displaystyle}r<{$}}

\usepackage{makecell}
\usepackage{bm}
\usepackage{nicefrac}
\usepackage{graphicx}
\usepackage[margin=3cm]{geometry}
\usepackage{hyperref}
\hypersetup{
    colorlinks=true,
    linkcolor=blue,
    citecolor=blue,
    urlcolor=blue
}
\usepackage[
  backend=biber,
  style=alphabetic,
  giveninits=true
]{biblatex}
\DeclareFieldFormat{doi}{%
  \textsc{doi}:\space
  \ifhyperref
    {\href{https://doi.org/#1}{\nolinkurl{#1}}}
    {\nolinkurl{#1}}}
\renewbibmacro*{doi+eprint+url}{%
  \printfield{doi}%
  \newunit\newblock
  \iftoggle{bbx:eprint}
    {\usebibmacro{eprint}}
    {}%
  \newunit\newblock
  \iffieldundef{doi}
    {\usebibmacro{url}}
    {}}

\newcommand{\Z}{\mathbb Z}
\newcommand{\R}{\mathbb R}

\newcommand{\Di}{\mathsf{D}}
\newcommand{\Ei}{\mathsf{E}}
\newcommand{\Ho}{\mathsf{Ho}}
\newcommand{\one}{\mathbf{1}}
\newcommand{\T}{^\mathsf{T}}

\DeclareMathOperator{\vol}{vol}
\DeclareMathOperator{\diag}{diag}
\DeclareMathOperator{\dens}{dens}
\newcommand{\sign}{\operatorname{sign}}

\theoremstyle{plain}
\newtheorem{theorem}{Theorem}[section]

\newtheorem{lemma}[theorem]{Lemma}
\newtheorem{question}[theorem]{Question}
\theoremstyle{definition}
\newtheorem{remark}[theorem]{Remark}

\AddToHook{env/longtable/begin}{\small\renewcommand{\arraystretch}{1.1}}

\author{Sven Ahrend}
\address{S. Ahrend, Universit\"at Rostock, Institute of Mathematics, 18051 Rostock, Germany}
\email{sven.ahrend@uni-rostock.de}
\author{Mathieu Dutour Sikiri\'c}
\address{M. Dutour Sikiri\'c, MSM Programing d.o.o., Karlovacka Cesta 28B, 10452 Klinca Selo, Croatia}
\email{mathieu.dutour@gmail.com}
\title{New best-known coverings, packing-coverings and quantizers in low dimensions}

\date{November 18, 2026}

\begin{document}

\begin{abstract}
We construct periodic point sets that improve the best known covering densities
in dimensions six and seven, packing-covering constants in dimensions five, six,
and seven, and normalized second moments in dimensions five and six.
In dimension five, a set of the original family of Cohn--Kumar--Schürmann has
packing-covering constant $9/\sqrt{40}$, strictly below the proven optimum
for the lattice case. 
All seven point sets belong to a two-parameter subfamily of a four-parameter 
generalization of the Cohn--Kumar--Schürmann family. 
\end{abstract}

\maketitle

%  ___       _                 _            _   _             
% |_ _|_ __ | |_ _ __ ___   __| |_   _  ___| |_(_) ___  _ __  
%  | || '_ \| __| '__/ _ \ / _` | | | |/ __| __| |/ _ \| '_ \ 
%  | || | | | |_| | | (_) | (_| | |_| | (__| |_| | (_) | | | |
% |___|_| |_|\__|_|  \___/ \__,_|\__,_|\___|\__|_|\___/|_| |_|
%                                                             
\section{Introduction}

In 2009 Cohn, Kumar, and Sch\"urmann \cite{CKS} introduced
the one-parameter family
\[
  \Di_n^+(\alpha) 
  =
  \left\{ 
  (z_1,\dots,z_{n-1},\alpha z_n)\T : 
  (z_1,\dots,z_n)\T \in \mathsf{D}_n^+ \right\},
  \qquad \alpha > 0,
\]
of two-periodic point sets.
Here
\[
  \Di_n^+ = 
  \Di_n\cup(t+\Di_n),
  \qquad 
  t=\dfrac12\one,
\]
where $\Di_n = \{ x \in \mathbb Z^n : \sum_{i=1}^n x_i \in 2 \mathbb Z\}$ is the
checkerboard lattice and $\one = (1,\dots,1)\T$ is the all-one vector.
For $n$ odd, since $t\in \Di_n^+$ but $2t=\one\notin \Di_n^+$, this set is not a lattice.
In this paper we generalize this point set to a four-parameter family of (in general)
four-periodic point sets $\Di_n^+(a_1,a_2,a_3; \alpha)$.
In Section~\ref{sec:four}, we define the new family of four-periodic point sets
related to the $\Di_n^+$ point set.
There we also define a two-periodic subfamily $\Di_n^+(g;\alpha)$.
This subfamily contains all the new records we report here.
As we show, for suitable parameter choices, the two-parameter family $\Di^+_n(g;\alpha)$ improves upon the best known values for several optimization problems in low dimensions.
\begin{itemize}
\item For the sphere covering problem in dimensions $n=6,7$.
\item For the sphere packing-covering problem in dimensions $n=5,6,7$, with the five-dimensional example
 beating the \emph{optimal} lattice value of $\Ho_5$.
\item For the quantization problem in dimensions $n=5,6$.
\end{itemize}
The above-mentioned optimization problems are defined in Section~\ref{sec:op}.
There we also summarize the current state of knowledge about optimal point sets.

We explicitly state the five-dimensional result because it improves upon the \emph{proven} optimal lattice value:
\begin{theorem}\label{thm:record}
The two-periodic set $\Di_5^+\left(\tfrac{2}{\sqrt{15}}\right)$ has packing-covering
constant $\gamma = \tfrac{9}{\sqrt{40}} = 1.423024\ldots < \gamma_5$.
\end{theorem}
This follows from substituting $\tau = \tfrac{4}{15}$ into the second row of Table~\ref{tab:packing-covering}.

The remainder of the paper is structured as follows.
In Section~\ref{sec:Hrep}, we define a polytope $P_\tau(g)$ by explicit linear
inequalities and show that it is the intersection of the Voronoĭ cell of
$\Di^+_n(g;\alpha)$ with a fundamental chamber of a symmetry group.
In Section~\ref{sec:Vrep}, for fixed dimension, we describe how to symbolically 
compute the vertices of $P_\tau(g)$. For fixed $g$ these are described by 
affine functions of $\tau$ in suitable coordinates, and a triangulation can also 
be computed symbolically.
Because dimension five contains the most important record, 
in Section~\ref{sec:objective} we give 
the exact covering density, packing-covering constant, and normalized second moment of 
$\Di_5^+\left(\sqrt{\tau}\right)$ for all $\tau > 0$, as piecewise algebraic functions.

These computations are related to Voronoĭ reduction theory for periodic point sets,
as developed in \cite{Schuermann}, and its application to the
quantizer problem in Zimmermann's thesis \cite{zimmermann}. 
For fixed $g$, our computations correspond
to this theory with a one-dimensional (affine) subspace.
For our particular family, we additionally obtain an explicit inequality description
of the Voronoĭ cell, up to symmetry, valid for variable $g$.

Section~\ref{sec:record} contains the parameter choices that give the claimed
improvements in low dimensions and the data on the closed forms. All of the records
can be verified with the supplementary code distributed with this document.

%   ___        _   _           _          _   _             
%  / _ \ _ __ | |_(_)_ __ ___ (_)______ _| |_(_) ___  _ __  
% | | | | '_ \| __| | '_ ` _ \| |_  / _` | __| |/ _ \| '_ \ 
% | |_| | |_) | |_| | | | | | | |/ / (_| | |_| | (_) | | | |
%  \___/| .__/ \__|_|_| |_| |_|_/___\__,_|\__|_|\___/|_| |_|
%       |_|                                                 
%  ____            _     _                    
% |  _ \ _ __ ___ | |__ | | ___ _ __ ___  ___ 
% | |_) | '__/ _ \| '_ \| |/ _ \ '_ ` _ \/ __|
% |  __/| | | (_) | |_) | |  __/ | | | | \__ \
% |_|   |_|  \___/|_.__/|_|\___|_| |_| |_|___/
%                                             
\section{Optimization problems}\label{sec:op}

Let $\|\cdot\|$ be the Euclidean norm on $\mathbb R^n$, the $n$-dimensional space of column vectors,
and let $B^n = \{ x \in \mathbb R^n : \| x \| \leq 1 \}$ be the unit ball.
Suppose $X$ is a discrete set.
The Voronoĭ cell of a point $x \in X$ is the set of all points
lying at least as close to $x$ as to any other point of $X$:
\[
 \{ z \in \mathbb R^n : \|x-z\| \leq \|y-z\| \text{ for all } y \in X\}.
\]
The defining inequality for a pair of points $x$ and $y$ in the definition is called
the Voronoĭ inequality and is a linear inequality:
\[
  \|x-z\| \leq \|y-z\| \,\Leftrightarrow 2\left(y-x\right)\T z \leq \|y\|^2 - \|x\|^2.
\]

For the remainder of the section, assume that all Voronoĭ cells are bounded and congruent
to the Voronoĭ cell $V$ of the origin. This restriction is satisfied by two-periodic point sets
(see Section~\ref{sec:objective}),
and is imposed solely because it simplifies notation.

Let
\[
  \rho(X) = \dfrac 1 2 \inf_{x \neq y \in X} \| x-y \|,
  \qquad
  \mu(X) = \sup_{z \in \R^n} \inf_{x \in X} \|z-x\|
\]
be the packing and covering radii. Note that the packing radius is the radius of 
the largest sphere centered at the origin that is contained in the Voronoĭ cell,
and likewise the covering radius is the radius of the smallest sphere centered at the
origin that contains the Voronoĭ cell.

Let $C = \{ x \in \R^n : \vert x_i \vert \leq \tfrac 1 2\}$ be the $n$-dimensional
cube of unit edge length. The point density can be defined as
\[
  \dens X = \lim_{\lambda \rightarrow \infty} \dfrac{ \vert X \cap \lambda C\vert}{\vol \lambda C},
\]
if this limit exists. Although there are many pathological examples motivating broader
definitions of density, this definition is adequate for the point sets we consider here.
With this
\[
  \delta(X) = \rho(X)^n \vol B^n \dens X
  \quad\text{and}\quad
  \Theta(X) = \mu(X)^n \vol B^n \dens X
\]
are the packing and covering densities.

The problem of minimizing $\Theta(X)$ is known as the sphere covering problem.
Its minimizer gives the most economical way to cover space by equally sized spheres. 
The optimal \emph{lattice} sphere covering density is known for dimension at most five,
while the unrestricted problem is solved for dimension at most two only.
The two-dimensional case is due to Kershner \cite{kershner}. 
The three-dimensional lattice case is due to Bambah \cite{bambah};
the four-dimensional one is due to Delone and Ryshkov \cite{DR}. 
The five-dimensional lattice case was proven in a series of papers by 
Ryshkov and Baranovskiĭ, which culminated in the monograph \cite{RB}. 
The lattice cases are also verified by Schürmann and Vallentin in \cite{SV}, 
where the authors give a computational proof.

The quotient of the covering radius and packing radius is known as the 
packing-covering constant 
\[
  \gamma(X) = \dfrac{\mu(X) }{\rho(X)}.
\]
Minimizing it amounts to seeking a set
that is simultaneously a good packing and a good covering; equivalently, to
finding the thinnest $(r,R)$-system in Ryshkov's terminology or the ``closest
packing'' in L.~Fejes T\'oth's terminology. After normalizing $\rho(X)=1$, the
supremal radius of an additional sphere that can be packed among the unit
spheres centered at $X$ is $\mu(X)-1=\gamma(X)-1$. 

As in the sphere covering problem, the optimum among lattices is known for $n \leq 5$.
The two-dimensional case is due to Ryshkov \cite{Rys74}, the three-dimensional one to
B\"or\"oczky \cite{Boroczky} --- who proved it \emph{without} the restriction
to lattices --- and the four- and five-dimensional cases to Horv\'ath \cite{Horvath82,Horvath86}. The
lattices $\Ho_4$ and $\Ho_5$ are neither best packing nor best covering
lattices. Sch\"urmann and Vallentin \cite{SV} reproduced the lattice result computationally.
For a concise and elegant exposition see \cite{Zong}. Recent improvements in the lattice
case have been reported in \cite{LZ}, which in particular shows that the $\Ei_8$ root lattice
is not optimal for the packing-covering problem. 

The normalized second moment is defined as
\[
  G(X)
  =
  \dfrac{1}{n}
  \dfrac{
    \int_V \|x\|^2 \,\mathrm{d}x
  }{
    \left( \vol V \right)^{1+2/n}
  },
\]
and is also known as the quantizer constant; minimizing it is the quantizer problem.
The optimal lattice quantizer constant is known for dimension at most four. 
The three-dimensional case is due to Barnes and Sloane \cite{BarnesSloane},
who showed that the BCC lattice is optimal.
The conjecture that this is in fact optimal without the restriction to lattices
is related to \emph{Gersho's conjecture}.
For a recent breakthrough in this direction, see the two independent papers \cite{crystallization}
and \cite{stability}. The former also contains a detailed historical survey and 
an extensive bibliography, to which we refer the interested reader. 
The four-dimensional lattice case was proved by Liu \cite{liu}.
See also \cite{AA-best} for recent work improving on many of the best known lattice quantizers
in dimensions up to $48$. Further improvements can be found in \cite{AAP-glued,AAP-exact,AAP-gradient}.

We also mention the seminal work of Agrell and Eriksson \cite{AE}, two of whose results are highly 
relevant here. First, the authors observe that $\Di_n^+$ yields better quantizers
than the best lattices known in dimensions seven and nine. Second, 
they exhibit a nine-dimensional ``lattice, in which the lattice points lie closer
together along the [last] coordinate than along the others''. These are precisely
the ingredients behind the records presented here. See also \cite{AA-dim9}, where 
a closed form for this lattice is derived.

%  _  _                         _           _ _      
% | || |        _ __   ___ _ __(_) ___   __| (_) ___ 
% | || |_ _____| '_ \ / _ \ '__| |/ _ \ / _` | |/ __|
% |__   _|_____| |_) |  __/ |  | | (_) | (_| | | (__ 
%    |_|       | .__/ \___|_|  |_|\___/ \__,_|_|\___|
%              |_|                                   
\section{A family of four-periodic point sets}\label{sec:four}

We give a construction of four-periodic sets in dimension $d = n+1$ from 
the $n$-dimensional lattice $\Di_n$.
From this section onward, $d=n+1$ denotes the ambient dimension and $n$ the
dimension of a horizontal layer.
Write $x = (x_1,\dots,x_n,z)$ and $t = \tfrac12 \one_d$.
Then $\Di_d^+ = \Di_d \cup (t + \Di_d)$, where:
\begin{itemize}
  \item $x \in \Di_d$ if and only if $x_i \in \Z$, $z \in \Z$ and $\sum_i x_i + z \in 2\Z$.
  \item $x \in t + \Di_d$ if and only if $x_i \in \tfrac12 + \Z$, $z \in \tfrac12 + \Z$ and $\sum_i x_i + z \in \tfrac{d}2 + 2\Z$.
\end{itemize}

The horizontal layers repeat with period $2$ in $z$.
\begin{itemize}
  \item $z = 0$: $\sum_i x_i \in 2\Z$, so the layer is $\Di_n \times \{0\}$.
  \item $z = \tfrac12$: $\sum_i x_i \in \tfrac n 2 +  2\Z$, so the layer is $(\tfrac12\one_n + \Di_n) \times \{\tfrac12\}$.
  \item $z = 1$: $\sum_i x_i \in 1 + 2\Z$, so the layer is $(e_1 + \Di_n) \times \{1\}$.
  \item $z = \tfrac32$: $\sum_i x_i \in \tfrac n 2 + 1 + 2\Z$, so the layer is $(e_1 + \tfrac12\one_n + \Di_n) \times \{\tfrac32\}$.
\end{itemize}

Varying the vertical distance between the layers yields a four-periodic set:
Let $a_0 = 0$ and let $a_1,a_2,a_3 \in [0,2)$ be real numbers. With
\[
  C_0 = \Di_n,\quad 
  C_1 = \dfrac12\one_n + \Di_n,\quad
  C_2 = e_1 + \Di_n,\quad 
  C_3 = e_1 + \dfrac12\one_n + \Di_n
\]
(the four cosets of $\Di_n$ in its dual $\Di_n^*$), define
\[
  \Di^+_d(a_1,a_2,a_3) = \bigcup_{j=0}^{3} C_j \times (a_j + 2\Z) \subset \R^d.
\]
In particular
$\Di_d^+ = \Di^+_d(\tfrac12,1,\tfrac32)$.

For $\alpha > 0$ define $\Di^+_d(a_1,a_2,a_3;\alpha)$ as
\[
  \{ (x_1,\dots,x_n,\alpha z) : (x_1,\dots,x_n,z) \in \Di^+_d(a_1,a_2,a_3) \}.
\]

Consequently, $\Di_d^+(\alpha) = \Di_d^+(\tfrac 1 2, 1, \tfrac 32;\alpha)$,
hence this generalizes the family of two-periodic sets defined by 
Cohn, Kumar, and Sch\"urmann in \cite{CKS}.

\begin{remark}
Note that for $n\geq 3$ the (nontrivial) cosets of $\Di_n$ in $\Di_n^*$ are the 
holes of $\Di_n$. Therefore the construction is: 
Place cosets \emph{above} the holes of the lattice, where the
vertical distance is controlled by the parameters $a_i$ and $\alpha$.
This construction could be applied to any point set in dimension $n$ and any subset of holes
to define a new point set in dimension $n+1$.
\end{remark}

Write $\Di_d^+(g;\alpha) = \Di_d^+(g,1,1+g;\alpha)$.
Note that we have
\[
  \Di_d^+(g,1,1+g) = \Di_d \cup (t_g + \Di_d),
  \qquad
  \text{ where } t_g = \begin{pmatrix} \dfrac 1 2\one \\ g\end{pmatrix} \in \mathbb R^{d}.
\]

Let $H_\tau$ be a positive definite quadratic form (specified below) and set $\tau = \alpha^2$.
Instead of scaling the point set along the last coordinate, we work with the fixed
point set $\Di_d^+(g;1)$ and vary the metric defined by $H_\tau$. Therefore we study
\[
  X_g
  =
  \Di_d\cup(t_g+\Di_d),
  \qquad 
  H_\tau=\diag(1,\dots,1,\tau).
\]

We can restrict to $g \in [0,\tfrac 1 2]$. Indeed, $\Di^+_d(g;\alpha)$,
$\Di^+_d(g+1;\alpha)$, and $\Di^+_d(-g;\alpha)$ are isometric. 
To see this, note that the sign change $x_1 \mapsto -x_1$
maps $t_g$ to $t_g -e_1$, and $t_g - e_1 \in t_{g+1} + \Di_d$. 
Because this map is a symmetry of both $\Di_d$  and $H_\tau$, we have
that $\Di^+_d(g;\alpha)$ and $\Di^+_d(g+1;\alpha)$ are isometric. 
Similarly the reflection $z \mapsto -z$ maps $t_g$ to $t_{-g}$ and 
gives the isometry between $\Di^+_d(g;\alpha)$ and $\Di^+_d(-g;\alpha)$.

\begin{lemma}
Let $g \in [0,\tfrac 1 2]$. Then the packing radius $\rho_\tau$ of $X_g$ with 
respect to the norm induced by $H_\tau$ satisfies 
\[
  4\rho^2_\tau = \min \left\{ 2,1+\tau,4\tau, \tfrac n4+g^2\tau \right\}.
\]
\end{lemma}
\begin{proof}
The proof is elementary.
First note that a nonzero vector of $\Di_d$ either has only its last coordinate nonzero, in
which case that coordinate is an even integer, or has a nonzero coordinate
among the first $n$. These cases yield
\[
 \min \{v\T H_\tau v : 0 \neq v \in \Di_d  \}=\min\{4\tau,1+\tau,2\}.
\]
Indeed, an odd coordinate among the first $n$ must be accompanied by another odd
coordinate, either among the first $n$ or in the last position. A nonzero
even coordinate among the first $n$ contributes at least $4$ to the squared
norm. These observations prove the lower bound; $2e_d$, $e_1+e_d$, and
$e_1+e_2$ attain the three values.

Every difference between the two cosets has half-integral values in all of its first
$n$ coordinates. The value of the last coordinate is in $\pm g + \Z$, hence its
absolute value is at least $g$ for $0\leq g \leq \tfrac 1 2$.
Therefore the squared norm of any difference vector 
is at least $\tfrac{n}{4}+g^2\tau$, attained by $t_g$. 
This proves the claimed packing radius.
\end{proof}
%  _   _                      
% | | | |      _ __ ___ _ __  
% | |_| |_____| '__/ _ \ '_ \ 
% |  _  |_____| | |  __/ |_) |
% |_| |_|     |_|  \___| .__/ 
%                      |_|    
\section{The Voronoĭ cell}\label{sec:Hrep}

From now on we assume $n\geq 3$ and $0\le g\le\tfrac12$.

Permute the first $n$ coordinates and apply sign changes $x_i \mapsto \varepsilon_ix_i$, 
where $\varepsilon_i \in \{\pm 1\}$ and $\prod_{i=1}^n \varepsilon_i =1$. These
transformations preserve $X_g$. They give a subgroup of order $n!\,2^{n-1}$, and any
point can be moved into the region
\[
 \mathcal C_D=\{(x_1,\dots,x_n,z)\T:x_1\ge x_2\ge \dots \ge x_{n-1}\ge|x_n|\},
\]
where $z$ is unrestricted. Every such transformation preserves the squared norm.

For the moment let $H$ be any positive definite quadratic form. 
The Voronoĭ inequality associated with a point $0\neq v \in X_g$ is
\[
 2v\T Hx\le v\T Hv.
\]

We use the following elementary observation:
Let $u,v$ be two points with $u\T Hv \geq 0$, and suppose that $x$ satisfies the Voronoĭ
inequalities associated with $u$ and $v$. Now consider the Voronoĭ inequality of their sum $u+v$:
\[
  2(u+v)\T H x \leq u\T H u + v\T H v = (u+v)\T H (u+v) - 2u\T H v.
\]
Since $u\T Hv\ge0$, the right-hand side is at most $(u+v)\T H(u+v)$, so the
Voronoĭ inequality of $u+v$ is redundant. We will apply this inductively.

We now return to $H = H_\tau$. Denote the Voronoĭ cell of the origin by $V_\tau(g)$.
Use the nine points
\begin{gather*}
 \pm 2e_{d},\quad  (\tfrac12,\dots,\tfrac12,-\tfrac12,g-1)\T,\quad
 (\tfrac12,\dots,\tfrac12,g-2)\T,\\
 (\tfrac12,\dots,\tfrac12,g)\T,\quad
 (\tfrac12,\dots,\tfrac12,-\tfrac12,g+1)\T,\\
 e_1 \pm e_{d},\quad e_1+e_2.
\end{gather*}
They all belong to $X_g$. Set $S = x_1 + \dots + x_n$. The corresponding Voronoĭ inequalities, together with
the inequalities defining $\mathcal C_D$, define the following family of polytopes
$P_\tau(g)$ containing $V_\tau(g)\cap\mathcal C_D$:
\begin{align}
 -1\le z&\le1,\label{eq:p1}\\
 S-2x_n+2(g-1)\tau z&\le\dfrac n4+(g-1)^2\tau,\label{eq:p2}\\
 S+2(g-2)\tau z&\le\dfrac n4+(g-2)^2\tau,\label{eq:p3}\\
 S+2g\tau z&\le\dfrac n4+g^2\tau,\label{eq:p4}\\
 S-2x_n+2(g+1)\tau z&\le\dfrac n4+(g+1)^2\tau,\label{eq:p9}\\
 x_1 \pm \tau z & \le\dfrac{\tau}{2} + \dfrac 1 2,\label{eq:p5}\\
 x_1+x_2 & \le1,\label{eq:p7}\\
 x_1\ge x_2\ge \dots \ge x_{n-1}&\ge|x_n|.\label{eq:p8}
\end{align}
The inequalities give $n+9$ halfspaces. The polytope is bounded: $|z|\le1$ and $|x_i|\le1$.
Here we have used $\tau>0$ to obtain the bounds
$-1\le z\le1$ by dividing the corresponding
Voronoĭ inequalities by $4\tau$.

\begin{lemma}
For all $\tau >0$ and $g \in [0,\tfrac 1 2]$ we have $P_\tau(g) = V_\tau(g) \cap \mathcal C_D$.
\end{lemma}

\begin{proof}
Because the Voronoĭ inequalities defining $P_\tau(g)$ are a subset of 
all Voronoĭ inequalities, we have $V_\tau(g) \cap \mathcal C_D\subseteq P_\tau(g) $.
Now we show equality. We do this by showing that all other Voronoĭ inequalities are
redundant.

We start by considering points in the first coset $\Di_{d}$.
From the reduction above it follows that all Voronoĭ inequalities defined by points that are not
of the form $\pm e_i \pm e_j$ allowing $i = j$ are redundant. 
To see this, take any nonzero $a \in \Di_{d}$. Since the coordinate sum is even, 
$\sum_i \vert a_i \vert$ is also even, say equal to $2m$. 
Now define a sequence $f_i$ of $2m$ vectors with the property that
\[
  \vert a_k\vert
  =
  \vert \{ i : f_i = \sign(a_k) e_k\} \vert
  \qquad {k = 1,\dots,d}.
\]
Now by construction we have that $\sum_{i=1}^{2m} f_i = a$ and that $f_i\T f_j \geq 0$
for each $i,j \in \{1,\dots,2m\}$. In other words, for 
each component the entries never have opposite signs. This property extends to the
vectors defined by $b_i = f_{2i-1} + f_{2i}$ for $i = 1,\dots,m$. Note that each $b_i$
is of the form $\pm e_k \pm e_l$ (possibly $k=l$), and is nonzero since
$f_{2i-1}\T f_{2i}\ge 0$. Finally, since
$H_\tau$ is diagonal with positive entries, we have $b_i\T H_\tau b_j \geq 0$, and the reduction
applies.

Next we show that all of the Voronoĭ inequalities of points $\pm e_i \pm e_j$ 
are either redundant or defining inequalities of $P_\tau(g)$.
Write $x = (x_1,\dots,x_n,z)\T$ and let $s,s' \in \{\pm 1\}$. Then
\begin{itemize}
  \item For $1 \leq i < j \leq n$: The inequality is $2(sx_i + s'x_j) \leq 2$. It
 follows from $sx_i + s'x_j \leq |x_i| + |x_j| \leq x_1 + x_2 \leq 1$. Here $x_2 \geq 0$ because $n\geq 3$.
  \item For $1 \leq i < j = d$: The inequality is $2(sx_i + s'\tau z) \leq 1 + \tau$.
 It follows from $s x_i + s' \tau z \leq x_1 + s' \tau z \leq \tfrac{1 +\tau}{2}$. 
 The last step uses the inequalities \eqref{eq:p5}. 
  \item $1 \leq i = j \leq n$: The inequality is $\pm 4x_i \leq 4$ and it is implied by 
  $|x_i| \leq x_1 \leq 1$.
  \item $i=j=d$: The inequalities associated with $\pm 2e_{d}$ are defining inequalities of $P_\tau(g)$.
\end{itemize}

The case $a\in t_g+\Di_{d}$ is similar. Here any vector may be reduced to one with
entries in $\pm\tfrac12$ in the first $n$ coordinates and last coordinate $g+k$,
where $k\in\{-2,-1,0,1\}$.

To see this, let $a \in t_g + \Di_{d}$. Suppose $\vert a_i \vert \geq 2$
for some $i$. Then set 
\[
  r = 2 \sign(a_i)e_i \qquad 
  u = a-r \in t_g+\Di_{d}.
\]
Now $u\T H_\tau r = 2h_{ii}\left(\vert a_i\vert - 2\right)\geq0$, where $h_{ii}$ is the
$i$-th diagonal entry of $H_\tau$. Therefore we can reduce to
$a_i\in\left\{\pm\tfrac12,\pm\tfrac32\right\}$ for $i\le n$ and $|a_{d}|<2$.
Since $a_d\in g+\Z$ and $0\le g\le\tfrac12$, this gives the stated choices of $k$.

We now turn to the first $n$ coordinates. Let $i \in \{1,\dots,n\}$ be minimal
with $\vert a_i \vert \geq \tfrac3 2$ and $j\in\{1,\dots,n\}$ the smallest index different from $i$.
Define
\[
  r = \sign(a_i) e_i + \sign(a_j) e_j \qquad
  u = a-r \in t_g+\Di_{d}.
\]
With this we have
\[
  u\T H r =
  \left( \vert a_i \vert -1 \right)
  +
  \left( \vert a_j \vert -1 \right)
  \geq 0,
\]
because $\vert a_i\vert \geq \tfrac 3 2$ and $\vert a_j\vert \geq \tfrac 1 2$.
This reduces $\vert a_i \vert$ by $1$ and therefore this construction eventually
terminates with $\vert a_i \vert = \tfrac 1 2$ for all $i=1,\dots,n$. 

Now write $a=(\tfrac{\varepsilon_1}{2},\dots,\tfrac{\varepsilon_n}{2},g+k)\T$ with
$\varepsilon_i\in\{\pm1\}$ and $k\in\{-2,-1,0,1\}$, and put
\[I=I(a)=\{i=1,\dots,n:\varepsilon_i=-1\}.\]
Here $a\in t_g+\Di_{d}$ is equivalent to
$|I|\equiv k\pmod2$. Indeed among the first $n$ coordinates $a-t_g$ has exactly
$|I|$ non-zero entries equal to $-1$ while the last coordinate is $k$.

On $\mathcal C_D$ we have $\sum_{i\in I}x_i\ge0$ when $|I|$ is even, and
$\sum_{i\in I}x_i\ge x_n$ when $|I|$ is odd. Indeed all $x_i$ with $i<n$ are
nonnegative and at least $|x_n|$.

Now we show that all of the Voronoĭ inequalities of the remaining points
are either redundant or defining inequalities of $P_\tau(g)$.
\begin{itemize}
  \item If $a_{d}=g$, then $|I|$ is even and
  \[S-2\sum_{i\in I}x_i+2g\tau z\le S+2g\tau z\le\tfrac n4+g^2\tau.\]
  The last step uses inequality \eqref{eq:p4}.
  \item If $a_{d}=g-1$, then $|I|$ is odd and
  \[S-2\sum_{i\in I}x_i+2(g-1)\tau z\le S-2x_n+2(g-1)\tau z\le\tfrac n4+(g-1)^2\tau.\]
  The last step uses inequality \eqref{eq:p2}.
  \item If $a_{d}=g+1$, then $|I|$ is odd and
  \[S-2\sum_{i\in I}x_i+2(g+1)\tau z\le S-2x_n+2(g+1)\tau z\le\tfrac n4+(g+1)^2\tau.\]
  The last step uses inequality \eqref{eq:p9}.
  \item If $a_{d}=g-2$, then $|I|$ is even and
  \[S-2\sum_{i\in I}x_i+2(g-2)\tau z\le S+2(g-2)\tau z\le\tfrac n4+(g-2)^2\tau.\]
  The last step uses inequality \eqref{eq:p3}.
\end{itemize}

Hence $V_\tau(g)\cap\mathcal C_D=P_\tau(g)$ for all $\tau>0$, $0\le g\le\tfrac12$
and $n\ge3$. 
\end{proof}

For $g=\tfrac12$, the additional symmetry
$(x_n,z)\mapsto(-x_n,-z)$ doubles the group order to $n!\,2^n$ and allows the
smaller chamber $\mathcal C=\{x_1\ge\dots\ge x_n\ge0\}$.
On this chamber, \eqref{eq:p9} follows from \eqref{eq:p4}, $z\le1$ and $x_n\ge0$.
For the $g=\tfrac12$ computations below, write $V_\tau=V_\tau(\tfrac12)$ and
$P_\tau=P_\tau(\tfrac12)\cap\mathcal C$.

% __     __    ____            
% \ \   / /   |  _ \ ___ _ __  
%  \ \ / /____| |_) / _ \ '_ \ 
%   \ V /_____|  _ <  __/ |_) |
%    \_/      |_| \_\___| .__/ 
%                       |_|    
\section{Vertices and triangulation}\label{sec:Vrep}

For a fixed dimension and fixed $0 \leq g \leq \tfrac 1 2$, we can compute the 
vertices as affine functions of $\tau$ in the coordinates $(x_1,\dots,x_n,w)$.
We spell out the details for the general
case and give the computational data for dimension 5, 
$0 < \tau < \tfrac 2 3$ and $g = \tfrac 1 2$.

Substitute $w = \tau z$ and multiply the inequality $-1\le z \le 1$ by $\tau$.
Then the coefficient matrix $A$ of the linear system of inequalities is independent of $\tau$.
The remaining $\tau$ dependence is on the right-hand side $b_\tau = b_0 + \tau b_1$.
Every vertex of a full-dimensional polytope in $d$ dimensions has $d$ linearly 
independent active defining inequalities.

Let $I$ be such a $d$-element index set and denote the subsystem corresponding
to it with a subscript $I$.  We have two systems
\[
  A_I c_0 = b_{0,I} \text{ and } A_I c_1 = b_{1,I}.
\]
In the coordinates $(x_1,\dots,x_n,w)$ the candidate vertex is then $v(\tau) = c_0 + \tau c_1$. 
For this candidate to be an actual vertex, we still require that
\[
  b - Av = \left( b_0 - Ac_0 \right) + \tau\left(b_1 - Ac_1\right) \geq 0.
\]
This bounds $\tau$ for each candidate vertex.
Each component gives an upper bound, a lower bound, or a condition independent of $\tau$,
according to the sign of the corresponding component of $b_1 - Ac_1$. 
These bounds partition the range $\tau \in (0,\infty)$ into intervals.
In the interior of these intervals the combinatorial structure of $P_\tau(g)$ is constant,
because the set of active inequalities is constant for each vertex.
For the optimization problems we consider, the finitely many points where two adjacent
intervals meet need no separate consideration: there the value follows from
either formula by continuity of the objectives.

Note that not all of the vertices are vertices of the Voronoĭ cell $V_\tau(g)$,
because some arise by the intersection with $\mathcal C_D$.

\begin{remark}
If one keeps $g$ as a parameter, then $A$ depends linearly on $g$, 
and $b_1$ is quadratic in $g$. Then exactly as above the vertices are affine in $\tau$
and rational in $g$ (in coordinates $(x_1,\dots,x_n,w)$).
\end{remark}

\begin{lemma}\label{lem:finite}
For $d=5$, $g=\tfrac12$, and $0 < \tau < \tfrac 2 3$, the polytope $P_\tau$ has the 41 vertices listed in 
Table~\ref{tab:vert} in Appendix~\ref{app}. 
Furthermore the face lattice for this range is independent of $\tau$.
\end{lemma}

\begin{proof}
The proof is computational.
The point $\tfrac{1}{100}(4,3,2,1,0)\T$ satisfies all twelve inequalities strictly, so $P_\tau$ is full-dimensional. There are $\binom{12}{5}=792$ possible 
five-element sets for the defining inequalities of a vertex. 
Exactly 455 have non-singular coefficient matrices. 
There are 41 vertices and the range is $0 < \tau < \tfrac 2 3$.
The face lattice is constant for this range by construction.
\end{proof}

Table~\ref{tab:vrep} contains the combinatorial data for $g=\tfrac 1 2$ in dimension 5.
A row contains the relevant range of $\tau$, the number of vertices (resp. facets) 
of $P_\tau$ and finally the number of simplices in the triangulations used below.

\begin{table}
\centering
\renewcommand{\arraystretch}{1.2}
\begin{tabular}{CRRR}
\toprule
\tau & \text{Vertices} & \text{Facets} & \text{Simplices} \\
\midrule
 \left(0, \tfrac 2 3\right) & 41 & 12  & 130\\
 \left(\tfrac 2 3, \tfrac 4 5\right) & 46 & 12  & 135 \\
 \left( \tfrac 4 5, 1\right) & 40 & 12  & 96 \\
 (1, 2) & 24 & 9  & 43 \\
 (2, \infty) & 10 & 7  & 5\\
\bottomrule
\end{tabular}
\caption{Data on the polytope $P_\tau$ in dimension five, on each of the listed open intervals.}
\label{tab:vrep}
\end{table}

We fix a vertex ordering by the lexicographical ordering of the coordinates for
$\tau = \tfrac{4}{15}$. We then compute the pulling triangulation $\mathcal T_\tau$ at this
value.
This is a triangulation for every $0 < \tau < \tfrac{2}{3}$ because the combinatorial
type is fixed for this interval and the pulling triangulation depends only on 
the face lattice of the polytope (and the vertex ordering). For a definition and
proof, see, for instance, \cite[Theorem~5.7.3]{BeckSanyal} and the paragraph thereafter.
For the other intervals we proceed analogously, with a reference value of $\tau$
inside each interval.

%   ___  _     _           _   _           
%  / _ \| |__ (_) ___  ___| |_(_)_   _____ 
% | | | | '_ \| |/ _ \/ __| __| \ \ / / _ \
% | |_| | |_) | |  __/ (__| |_| |\ V /  __/
%  \___/|_.__// |\___|\___|\__|_| \_/ \___|
%           |__/                           
\section{Computing the objective}\label{sec:objective}

Let $g$ be fixed.

For any two-periodic point set $X = L + \{0,t\}$ with a period lattice $L$, all Voronoĭ cells
are congruent. The isometry is given by the point reflection across $\tfrac 1 2 t$
which is $x \mapsto t-x$.

The covering radius is therefore equal to the maximum of the norms
of the vertices of the Voronoĭ cell. The vertices of $P_\tau(g)$ that
are not vertices of the Voronoĭ cell can only have a smaller norm.
Therefore it suffices to compute the maximum of the norms of all of the vertices obtained
above.

Integrals and volumes below are expressed in the coordinates
$x=(x_1,\dots,x_n,z)\T$ and use the measure $\mathrm{d}x$ unless noted otherwise.
Under the coordinate transformation $y=H_\tau^{1/2}x$, we have
\[
  \|y\|^2=x\T H_\tau x,
  \qquad
  \mathrm{d}y=\sqrt{\tau}\,\mathrm{d}x.
\]

Let $\widetilde V_\tau(g)=H_\tau^{1/2}V_\tau(g)$ be the Euclidean
Voronoĭ cell of the origin in $H_\tau^{1/2}X_g$, and set
\[
  I_\tau=\int_{V_\tau(g)}x\T H_\tau x\,\mathrm{d}x.
\]
Since the two Voronoĭ cells per fundamental domain are
congruent and $\Di_d$ has covolume $2$, we have
$\vol V_\tau(g)=1$.
With the change of variables $y=H_\tau^{1/2}x$ this gives
\[
  \vol\widetilde V_\tau(g)=\sqrt{\tau},
  \qquad
  \int_{\widetilde V_\tau(g)}\|y\|^2\,\mathrm{d}y
  =\sqrt{\tau}\,I_\tau.
\]
Consequently,
\[
  G(\tau)
  =\dfrac{\sqrt{\tau}\,I_\tau}
         {d(\sqrt{\tau})^{1+2/d}}
  =\dfrac{I_\tau}{d\tau^{1/d}}.
\]

Let $T$ be a full-dimensional simplex $T = \operatorname{conv}\{v_0,\dots,v_d\}$.
Set $\hat v = \sum_{i=0}^d v_i$. It is well known (see, for instance,
\cite[Chap.~21, Eq.~(11)]{SPLAG}) that
\[ 
  \int_T x\T H x\,\mathrm{d}x
  = 
  \dfrac{ \vol T}{(d+1)(d+2)}
  \left(
  \hat v\T H \hat v
  +
  \sum_{i=0}^d
  v_i\T Hv_i
  \right).
\]
The second moment of $P_\tau(g)$ can therefore be computed from the triangulation
$\mathcal T_\tau(g)$: by additivity of the integral,
\[
  \int_{P_\tau(g)} x\T H_\tau x\,\mathrm{d}x
  =
  \sum_{T \in \mathcal T_\tau(g)} \int_T x\T H_\tau x\,\mathrm{d}x.
\]

We now specialize to  $d=5$, $g=\tfrac 1 2$, and $0<\tau<\tfrac 2 3$.
Note that therefore we now work with $P_\tau$, its triangulations and the larger
symmetry group of order $4!\,2^4 = 384$.
The integral $\int_{P_\tau} x\T H_\tau x\,\mathrm{d}x$ equals
\[
  -\dfrac{
  7\tau^5 + 360 \tau^4 -3840\tau^3 + 10880 \tau^2 - 15360\tau - 9984
  }{17694720}.
\]
Recall that to obtain the integral over the Voronoĭ cell we still have to
multiply this expression by the order of the symmetry group.
\begin{equation*}
  G(\tau)
  =
  -\dfrac{
  7\tau^5 + 360 \tau^4 -3840\tau^3 + 10880 \tau^2 - 15360\tau - 9984
  }{230400 \tau^{1/5}}.
\end{equation*}

Differentiating gives
\[
  G'(\tau) 
  = - \dfrac{ N(\tau)}
  {48000 \tau^{6/5}},
\]
with
\[
  N(\tau)
  = 
  7\tau^5 + 285 \tau^4 - 2240 \tau^3 + 4080 \tau^2 - 2560 \tau + 416.
\]
$G'(\tau)$ vanishes if and only if $N(\tau)$ does. 

The normalized second moments for the other intervals are listed in Table~\ref{tab:G}.
The covering and packing radii, covering density, and packing-covering
constant are listed in Table~\ref{tab:packing-covering} and illustrated in Figure~\ref{fig:obj}.

\begin{table}
\small
\centering
\renewcommand{\arraystretch}{1.8}
\begin{tabular}{LL}
\toprule
\tau &  G\left(\Di_5^+(\sqrt{\tau})\right) \\
\midrule
\left(0,\tfrac 2 3\right)&
(-7\tau^5 - 360\tau^4 + 3840\tau^3 - 10880\tau^2 + 15360\tau + 9984)/(230400\tau^{1/5}) \\
\left(\tfrac 2 3, \tfrac 4 5 \right)&
(-2923\tau^6 + 11304\tau^5 - 15600\tau^4 + 6400\tau^3 + 6720\tau^2 + 12288\tau - 256)/(230400\tau^{6/5}) \\
\left(\tfrac 4 5,1\right)&
(857\tau^6 - 5136\tau^5 + 12900\tau^4 - 17600\tau^3 + 14520\tau^2 + 768\tau + 416)/(86400\tau^{6/5}) \\
(1,2)&
(-7\tau^6 + 48\tau^5 - 60\tau^4 - 320\tau^3 + 1560\tau^2 + 5952\tau - 448)/(86400\tau^{6/5})\\
(2,\infty)&
(5\tau^2 + 104\tau - 16)/(1200\tau^{6/5}) \\
\bottomrule
\end{tabular}
\caption{The normalized second moment $G$ of $\Di_5^+(\sqrt{\tau})$.}
\label{tab:G}
\end{table}

\begin{table}
\centering
\renewcommand{\arraystretch}{1.2}
\begin{tabular}{CCCLL}
\toprule
\tau & \mu^2 & 4\rho^2 & \gamma=\mu/\rho & \Theta \\
\midrule
 \left(0,\dfrac{1}{5}\right) & \dfrac{\tau^{2}}{4} + \dfrac{\tau}{16} + \dfrac{1}{2} & 4 \tau & \dfrac{1}{4}\sqrt{\dfrac{4 \tau^{2} + \tau + 8}{\tau}} & \dfrac{8 \pi^{2} }{15 \sqrt{\tau}}\left(\dfrac{\tau^{2}}{4} + \dfrac{\tau}{16} + \dfrac{1}{2}\right)^{5/2} \\
 \left(\dfrac{1}{5},\dfrac{4}{15}\right) & \dfrac{9 \tau^{2}}{16} + \dfrac{1}{2} & 4 \tau & \dfrac 1 4 \sqrt{\dfrac{9 \tau^{2} + 8}{ \tau}} & \dfrac{8 \pi^{2} }{15 \sqrt{\tau}}\left(\dfrac{9 \tau^{2}}{16} + \dfrac{1}{2}\right)^{5/2} \\
 \left(\dfrac{4}{15},\dfrac{2}{3}\right) & \dfrac{9 \tau^{2}}{16} + \dfrac{1}{2} & \dfrac{\tau}{4} + 1 & \sqrt{\dfrac{9 \tau^{2} + 8}{\tau + 4}} & \dfrac{8 \pi^{2} }{15 \sqrt{\tau}}\left(\dfrac{9 \tau^{2}}{16} + \dfrac{1}{2}\right)^{5/2} \\
 \left(\dfrac{2}{3},\dfrac{4}{5}\right) & \dfrac{21 \tau^{2}}{16} - \dfrac{5 \tau}{4} + 1 & \dfrac{\tau}{4} + 1 & \sqrt{\dfrac{21 \tau^{2} - 20 \tau + 16}{\tau + 4}} & \dfrac{8 \pi^{2} }{15 \sqrt{\tau}}\left(\dfrac{21 \tau^{2}}{16} - \dfrac{5 \tau}{4} + 1\right)^{5/2} \\
 \left(\dfrac{4}{5},1\right) & \tau^{2} - \tau + 1 & \dfrac{\tau}{4} + 1 & 4\sqrt{\dfrac{\tau^{2} - \tau + 1}{\tau + 4}} & \dfrac{8 \pi^{2} }{15 \sqrt{\tau}}\left(\tau^{2} - \tau + 1\right)^{5/2} \\
 (1,2) & \dfrac{(\tau + 1)^{2}}{4 \tau} & \dfrac{\tau}{4} + 1 & \dfrac{2(\tau + 1)}{\sqrt{\tau^{2} + 4 \tau}} & \dfrac{\pi^{2} (\tau + 1)^{5}}{60 \tau^{3}} \\
 (2,4) & \dfrac{\tau}{16} + 1 & \dfrac{\tau}{4} + 1 & \sqrt{\dfrac{\tau + 16}{\tau + 4}} & \dfrac{8 \pi^{2} }{15 \sqrt{\tau}}\left(\dfrac{\tau}{16} + 1\right)^{5/2} \\
 (4,\infty) & \dfrac{\tau}{16} + 1 & 2 & \sqrt{\dfrac{\tau}{8} + 2} & \dfrac{8 \pi^{2} }{15 \sqrt{\tau}}\left(\dfrac{\tau}{16} + 1\right)^{5/2} \\
\bottomrule
\end{tabular}
\caption{The covering and packing radii, as well as the packing-covering constant and covering density of $\Di^+_5\left(\sqrt{\tau}\right)$.}
\label{tab:packing-covering}
\end{table}

\begin{figure}[p]
\centering
\includegraphics{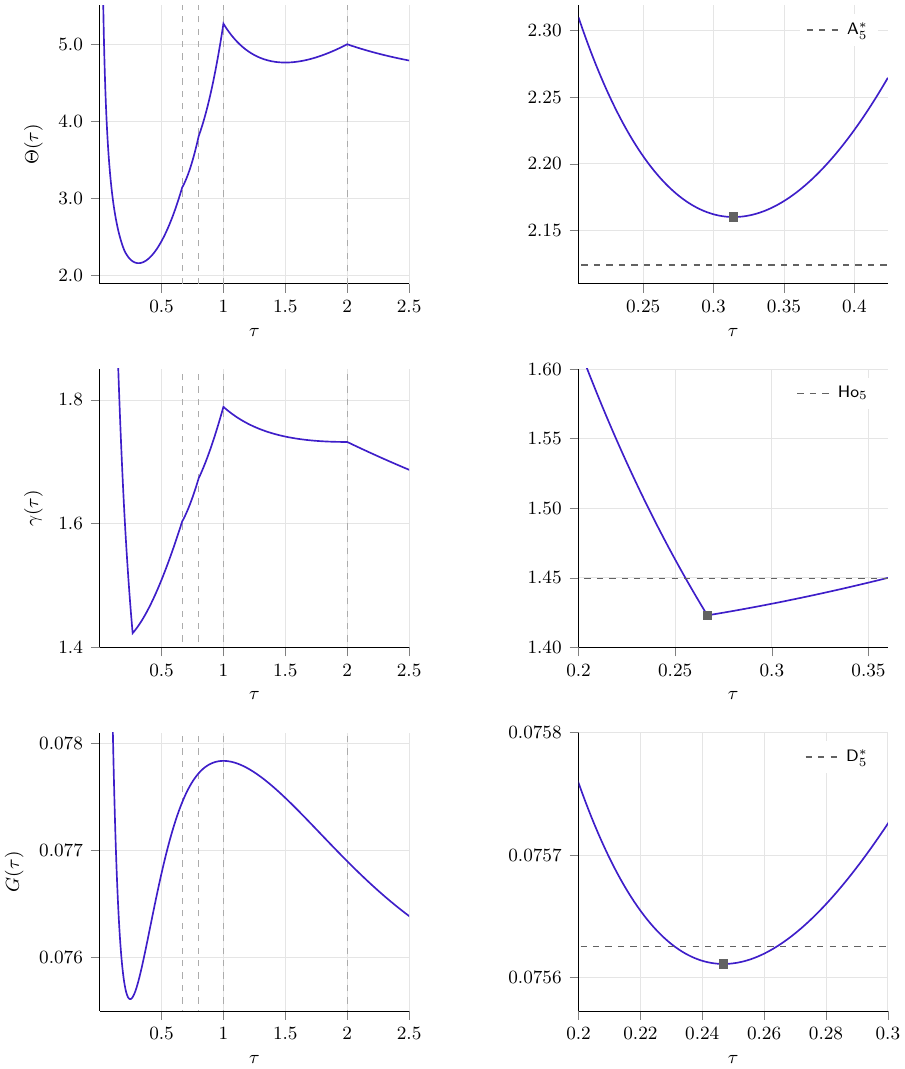}

\caption{%
  Covering density $\Theta(\tau)$ (top),
  packing-covering constant $\gamma(\tau)$ (middle),
  and normalized second moment $G(\tau)$ (bottom) 
  of $\Di_5^+\left(\sqrt{\tau}\right)$.
  The left column shows an overview; the right column shows
  a neighborhood of the respective minimum, which is marked by a square.
  Vertical dashed lines mark changes of the face lattice.
  The horizontal dashed line in each right-hand plot indicates
  the best (known) lattice value.
}
\label{fig:obj}
\end{figure}

%  _   _                 ____                        _     
% | \ | | _____      __ |  _ \ ___  ___ ___  _ __ __| |___ 
% |  \| |/ _ \ \ /\ / / | |_) / _ \/ __/ _ \| '__/ _` / __|
% | |\  |  __/\ V  V /  |  _ <  __/ (_| (_) | | | (_| \__ \
% |_| \_|\___| \_/\_/   |_| \_\___|\___\___/|_|  \__,_|___/
%                                                          
\section{New records in low dimensions}\label{sec:record}

The numerical values of the new records are listed in Table~\ref{tab:numerical}.
The exact parameter choices are listed in Table~\ref{tab:exact}. Combinatorial data
on the Voronoĭ-cells is recorded in Table~\ref{tab:combinatorial}. 
\begin{remark}
From the numerical representations in Table~\ref{tab:numerical} it may appear that
the packing-covering constant of the best previously known in dimension six coincides
with the currently best value in dimension seven. This is not the case: The 
packing-covering constants are
\[ 
  2\sqrt{2\sqrt{798} - 56}, 
  \qquad 
  \tfrac{1}{\sqrt{2}}\sqrt{\sqrt{12765} - 109}
\]
for the six-dimensional lattice due to Schürmann and Vallentin
\cite[Theorem~11.10]{SV} and the new non-lattice in dimension seven, respectively.
Both values round to $1.411081$ at six decimal places.
\end{remark}
In the remainder of the Section we describe how these parameter choices were found.

\begin{table}
\centering
\renewcommand{\arraystretch}{1.2}
\begin{tabular}{LLLLLLl}
\toprule
d & \text{\small{problem}} & g & \tau=\alpha^2 & \text{\tiny\makecell[l]{currently \\best value}} & \multicolumn{2}{c}{\tiny best previously known and source}  \\
\midrule
5 & \gamma & \tfrac12 & \tfrac4{15} & 1.4230249471 & 1.449457 &  \cite{Horvath86} \\
6 & \gamma & 0.3855477561 & 0.3245612707 & 1.4090378067 & 1.411081 &  \cite{SV} \\
7 & \gamma & \tfrac12 & 0.6374246304 & 1.4110810570 & 1.494983 &  \cite[p.~15]{Schuermann} \\
6 & \Theta & 0.2531514580 & 0.2744187299 & 2.3332897757 & 2.464801  & \cite{SV} \\
7 & \Theta & \tfrac12 & 0.5432906638 & 2.7345410245 & 2.900024 &  \cite{SV} \\
5 & G & \tfrac12 & 0.2469028095 & 0.0756111285 & 0.0756254 &  \cite{SPLAG} \\
6 & G & 0 & 0.3326579415 & 0.0742370363 & 0.0742437 &  \cite{SPLAG} \\
\bottomrule
\end{tabular}
\caption{Numerical values of the new records.}
\label{tab:numerical}
\end{table}

\begin{table}
\centering
\renewcommand{\arraystretch}{1.2}
\begin{tabular}{LLRRRRRRR}
\toprule
d & \text{\small{problem}} & 0 &1 &2 &3 &4 &5 &6 \\
\midrule
5 & \gamma & 2&4&8&6&5&-&-\\
6 & \gamma & 2&8&13&17&11&7&- \\
7 & \gamma & 8&13&16&13&12&7&5 \\
6 & \Theta & 2&8&13&17&11&7&- \\
7 & \Theta & 7&10&15&13&13&7&5 \\
5 & G &      2&4&8&6&5&-&- \\
6 & G &      2&4&7&10&7&5&- \\
\bottomrule
\end{tabular}
\caption{Combinatorial data on the Voronoĭ tessellation of the records.
The columns contain the number of face orbits under the full isometry
group of the point set. A column with label $i$ lists the number of orbits of
$i$-dimensional faces. 
Except for the six-dimensional quantizer all faces are primitive, 
i.e. each face of dimension $i$ is contained in exactly $d-i+1$ full-dimensional
cells of the Voronoĭ tiling. For the six-dimensional quantizer
there exist nonprimitive faces: For $i=0,1,2,3,4$ there are $2,1,2,1,1$ nonprimitive
face orbits, respectively. Each of those faces
is contained in exactly $d-i+2$ full dimensional Voronoĭ-cells.}
\label{tab:combinatorial}
\end{table}

With the methods described in the previous sections, one can obtain the closed formulas
for the covering density, the packing-covering constant and the normalized
second moment of the family $\Di_d^+(g;\alpha)$ as piecewise functions for fixed $g$
in small dimensions. Then one can compute the global minima within this fixed-$g$ family
by computing all local minima. Note that the minimum will be attained at an algebraic 
number if $g$ is rational.

For the full four-parameter family $\Di_d^+(a_1,a_2,a_3;\alpha)$ numerical optimization was performed as
follows: For a fixed choice of parameters, we computed the four types of Voronoĭ cells.
The Fincke--Pohst algorithm \cite{FP} was used to enumerate suitable vectors. 
From these cells, we computed the value of the objective function.
We then optimized the objective with the Nelder--Mead method \cite{NM}.
For quantization, we replace the cell moment and cell volume in the
definition of $G$ by their respective averages over the four cells,
with each moment taken about its generating point. For the covering radius
we take the maximum over all cells.

The values of the parameters converged to solutions 
in the family $\Di_d^+(g;\alpha)$. 
In dimension $6$ for the quantizer problem the resulting value was $g \approx 0$.
We then fixed $g=0$ and locally optimized the one-parameter family.
For the covering and packing-covering problems in dimension $6$ the resulting $g$ was
suspected to be irrational. We now describe how the closed forms are computed.

For the remainder of this section, let $d=6$.
Start with the numerical solutions $(g^*,\tau^*)$ obtained from the Nelder--Mead optimization.
Fix a rational value $g$ near $g^*$ and compute the decomposition of the interval $\tau >0$
as described in Section~\ref{sec:Vrep}. Take the open interval which contains $\tau^*$.
Take the set of active defining inequalities that define the vertices of $P_\tau(g)$
and compute the vertices in the coordinates $(x,w)$; they are rational functions
in $g$ and affine in $\tau$.

From this compute the squared norms of the vertices, which
are no longer affine in $\tau$. Now finding the maximum for the fixed $g^*$ 
amounts to finding the maximum of these functions. It turns out that for both 
the covering and packing-covering problems the maximum is given by the same 
formula. We have
\[
  \mu ^2 = M = M(g,\tau) = \dfrac{9}{16} + \dfrac{g^2}{2} \tau 
  + \left( g^4 + 2g^2 -2g +1\right) \tau^2,
\] 
near both numerical optima. 

Below we explain how to derive the claimed closed form. 
Computing the covering radius for these claimed algebraic parameters is then done
in a real algebraic number field, exactly as described in Section~\ref{sec:Vrep}.

We begin with the covering problem. Recall that the covering density is
$\Theta = \vol B^d \tfrac{\mu^d}{\sqrt{\tau}}$; squaring and taking the logarithm
yields that instead we minimize 
\[ 
  h =  d \log M - \log \tau = 2 \log \Theta - 2 \log \vol(B^d).
\]
We have 
\[
  \partial_g h = 0 \, \Leftrightarrow  \partial_g M = 0.
\]
Solving this for $\tau$ yields
\[
  \tau = \dfrac{g}{2-4g-4g^3}.
\]
The denominator cannot vanish, since this would force $g=0$,
at which the denominator equals $2$, a contradiction.
Requiring $\partial_\tau h = 0$ is equivalent to $6\tau \partial_\tau M = M$. This yields 
the following polynomial equation:
\[
  176g^4 \tau^2 + 352 g^2\tau^2 + 40 g^2\tau - 352 g \tau^2 + 176 \tau^2 -9 = 0.
\]
Substituting the expression for $\tau$ reduces this to
\begin{equation}
  A(g) = 32g^6+24g^4+32g^3-8g^2-36g+9 = 0.
\label{eq:A}
\end{equation}

We now turn to the packing-covering problem: Recall that $\gamma = \tfrac{\mu}{\rho}$.
Near the numerical optimum we have 
\[
  4 \rho^2 = \min \left\{ 4\tau, \dfrac{5}{4} + g^2 \tau \right\}.
\]
Therefore for the (squared) packing-covering constant we have the following
\[
  \gamma^2 = \max \left\{ \dfrac{M}{\tau}, \dfrac{4M}{\tfrac{5}{4} + g^2 \tau}\right\}.
\]
It turns out that this maximum is minimal when the two are equal. Indeed, when the
two terms are equal, we have
\[
  4\tau = \dfrac{5}{4} + g^2 \tau \qquad
  \Leftrightarrow
  \tau = \dfrac{5}{4(4-g^2)} = \tau_0.
\]
Then we have $\tau_0 \leq \tfrac 1 3$ because  $0\leq g \leq \tfrac 1 2$. 
Now set $c = g^4 + 2g^2 -2g +1$, for which we have $0 < c \leq 1$. 
Then compute the derivative $\partial_\tau \tfrac{M}{\tau} = c - \tfrac{9}{16\tau^2}< 0$,
for $0 < \tau \leq \tfrac 1 3$ because $c\leq 1$.
Similarly one can see that the second term is strictly increasing.
Therefore the maximum is minimal when the two are equal. 

Putting it together we have
\[
  \gamma^2 = 
  \dfrac{24g^4 + 18 g^2 - 50 g + 169}{20\left(4-g^2\right)}.
\]
Differentiation then yields the condition 
\begin{equation}
  B(g) = 24g^5-192g^3+25g^2-241g+100 = 0.
\label{eq:B}
\end{equation}

\begin{table}
\centering
\small
\renewcommand{\arraystretch}{1.4}
\begin{tabular}{LLLL}
\toprule
d & \text{\small{problem}}  & g & \tau \\
\midrule
5 & \gamma & \tfrac12 & \dfrac{4}{15} \\
6 & \gamma & B(g) & 5/(4(4-g^2)) \\
7 & \gamma & \tfrac12 & \tfrac{\sqrt{12765}}{69}-1 \\
6 & \Theta & A(g) & g/(2-4g-4g^3) \\
7 & \Theta & \tfrac12 & \tfrac{37+\sqrt{57}}{82} \\
5 & G & \tfrac12 & N(\tau) \\
6 & G & 0 & 26176\tau^6-66816\tau^5+54096\tau^4+2720\tau^3-27060\tau^2+15360\tau-2641 \\
\bottomrule
\end{tabular}
\caption{A row contains the 
dimension and the objective followed by data on $g$ and $\tau$. A value may
be defined by a polynomial; in this case it is the root closest to the value given
in Table~\ref{tab:numerical}. In the other cases the value is either explicitly given in radicals
or as a function of the other parameter. The polynomials $A$ and $B$ are defined in equations~\eqref{eq:A} and
\eqref{eq:B}, and $N$ in Section~\ref{sec:objective}.}
\label{tab:exact}
\end{table}

\section*{Open questions}

Our records occur in dimensions five, six, and seven. It is natural to ask
whether similar constructions improve the best known values in higher dimensions.
Determining the global minima of the three objectives among all point sets
in dimensions five, six, and seven also remains open.

Following Liu's solution of the four-dimensional lattice quantizer problem
\cite{liu}, the first unresolved dimensions for lattice quantization and lattice
covering are five and six, respectively. Our examples already improve
upon the best known lattice values in these dimensions. Proving that the best known lattices in these dimensions
are optimal would therefore establish that our examples
outperform every lattice for the corresponding problems.

Böröczky \cite{Boroczky} proved that the optimal three-dimensional lattice
packing-covering constant is also optimal among unrestricted point sets.
Our five-dimensional example shows that this no longer holds in dimension
five. A particularly natural remaining question is therefore:
\begin{question}
Does $\Ho_4$ minimize the packing-covering constant among all point sets
in $\mathbb R^4$?
\end{question}

\section*{Acknowledgements}
The first author thanks Erik Agrell for valuable correspondence concerning the quantizer problem.
Both authors thank Achill Schürmann for carefully reading the manuscript and for helpful comments. 

%     _    ___   ____  _        _                            _   
%    / \  |_ _| / ___|| |_ __ _| |_ ___ _ __ ___   ___ _ __ | |_ 
%   / _ \  | |  \___ \| __/ _` | __/ _ \ '_ ` _ \ / _ \ '_ \| __|
%  / ___ \ | |   ___) | || (_| | ||  __/ | | | | |  __/ | | | |_ 
% /_/   \_\___| |____/ \__\__,_|\__\___|_| |_| |_|\___|_| |_|\__|
%                                                                
\section*{Tool and computational resource disclosure}

This disclosure follows the Leiden Declaration on Artificial Intelligence and
Mathematics \cite{Leiden}.

\paragraph{Tools used.}
The initial exploration used a double-precision optimizer in
C\raisebox{0.2ex}{\scriptsize ++}, part of the \texttt{polyhedral} package
\cite{polyhedral}, using \texttt{libqhull} \cite{qhull} for the Delaunay tessellations.
A numerical candidate was refined to several hundred digits of precision,
and its minimal polynomials were identified using Python/\texttt{mpmath}.
Anthropic's Claude Opus 4.8 wrote those search and verification programs, ran the search campaigns,
and performed high-precision refinement and algebraic identification.
Later covering optimization and exploration of adjacent tessellations were
performed using Julia and \texttt{MOSEK}. OpenAI's ChatGPT-5.6 Sol assisted in recognizing the
$\Di_5$ description of the resulting Delaunay tessellation, which rediscovered
the $\Di^+_5(\alpha)$ family.
The symbolic computations needed in Section~\ref{sec:Vrep} and Section~\ref{sec:objective}
are implemented with SymPy \cite{SymPy}.
The Nelder--Mead optimization uses SciPy \cite{scipy}, Qhull \cite{qhull}, and NumPy \cite{numpy}.
An initial version of these implementations was written by OpenAI's ChatGPT-6 Sol. 
Writing assistance for this manuscript was limited to proofreading,
language suggestions, and typographical corrections.

\paragraph{Responsibility for correctness.}
Responsibility for the correctness and adequacy of the arguments and results,
and for the completeness and accuracy of the citations, rests with the human
authors, who designed the investigation and directed the strategy.
No automated system is an author of this paper.

\printbibliography
%     _                               _ _      
%    / \   _ __  _ __   ___ _ __   __| (_)_  __
%   / _ \ | '_ \| '_ \ / _ \ '_ \ / _` | \ \/ /
%  / ___ \| |_) | |_) |  __/ | | | (_| | |>  < 
% /_/   \_\ .__/| .__/ \___|_| |_|\__,_|_/_/\_\
%         |_|   |_|                            
\appendix
\section{Vertices of $P_\tau$}\label{app}

\begin{longtable}{R@{\quad}|R|R|R|R|R}
\caption{The coordinates of the $41$ vertices of $P_\tau$ in dimension five valid for $0 < \tau < \tfrac 2 3$.}
\label{tab:vert}\\
\toprule \text{No.} & x_1 & x_2 & x_3 & x_4 & z \\ \midrule
\endfirsthead
\toprule \text{No.} & x_1 & x_2 & x_3 & x_4 & z \\ \midrule
\endhead
\bottomrule\endfoot
1 & 0 & 0 & 0 & 0 & -1 \\
2 & 0 & 0 & 0 & 0 & 1 \\
3 & \tfrac{1}{4}-\tfrac{3}{16}\tau & \tfrac{1}{4}-\tfrac{3}{16}\tau & \tfrac{1}{4}-\tfrac{3}{16}\tau & \tfrac{1}{4}-\tfrac{3}{16}\tau & -1 \\
4 & \tfrac{1}{4}-\tfrac{3}{16}\tau & \tfrac{1}{4}-\tfrac{3}{16}\tau & \tfrac{1}{4}-\tfrac{3}{16}\tau & \tfrac{1}{4}-\tfrac{3}{16}\tau & 1 \\
5 & \tfrac{1}{3}-\tfrac{1}{4}\tau & \tfrac{1}{3}-\tfrac{1}{4}\tau & \tfrac{1}{3}-\tfrac{1}{4}\tau & 0 & -1 \\
6 & \tfrac{1}{3}-\tfrac{1}{4}\tau & \tfrac{1}{3}-\tfrac{1}{4}\tau & \tfrac{1}{3}-\tfrac{1}{4}\tau & 0 & 1 \\
7 & \tfrac{1}{4}+\tfrac{3}{16}\tau & \tfrac{1}{4}+\tfrac{3}{16}\tau & \tfrac{1}{4}+\tfrac{3}{16}\tau & \tfrac{1}{4}+\tfrac{3}{16}\tau & - \tfrac{1}{2} \\
8 & \tfrac{1}{3}+\tfrac{1}{12}\tau & \tfrac{1}{3}+\tfrac{1}{12}\tau & \tfrac{1}{3}+\tfrac{1}{12}\tau & 0 & 0 \\
9 & \tfrac{1}{3}+\tfrac{1}{12}\tau & \tfrac{1}{3}+\tfrac{1}{12}\tau & \tfrac{1}{3}+\tfrac{1}{12}\tau & \tfrac{1}{2}\tau & - \tfrac{1}{2} \\
10 & \tfrac{1}{2}-\tfrac{1}{2}\tau & 0 & 0 & 0 & -1 \\
11 & \tfrac{1}{2}-\tfrac{1}{2}\tau & 0 & 0 & 0 & 1 \\
12 & \tfrac{1}{2}-\tfrac{1}{2}\tau & \tfrac{1}{6}-\tfrac{1}{12}\tau & \tfrac{1}{6}-\tfrac{1}{12}\tau & \tfrac{1}{6}-\tfrac{1}{12}\tau & -1 \\
13 & \tfrac{1}{2}-\tfrac{1}{2}\tau & \tfrac{1}{6}-\tfrac{1}{12}\tau & \tfrac{1}{6}-\tfrac{1}{12}\tau & \tfrac{1}{6}-\tfrac{1}{12}\tau & 1 \\
14 & \tfrac{1}{2}-\tfrac{1}{2}\tau & \tfrac{1}{4}-\tfrac{1}{8}\tau & \tfrac{1}{4}-\tfrac{1}{8}\tau & 0 & -1 \\
15 & \tfrac{1}{2}-\tfrac{1}{2}\tau & \tfrac{1}{4}-\tfrac{1}{8}\tau & \tfrac{1}{4}-\tfrac{1}{8}\tau & 0 & 1 \\
16 & \tfrac{1}{2}-\tfrac{1}{2}\tau & \tfrac{1}{2}-\tfrac{1}{2}\tau & 0 & 0 & -1 \\
17 & \tfrac{1}{2}-\tfrac{1}{2}\tau & \tfrac{1}{2}-\tfrac{1}{2}\tau & 0 & 0 & 1 \\
18 & \tfrac{1}{2}-\tfrac{1}{2}\tau & \tfrac{1}{2}-\tfrac{1}{2}\tau & \tfrac{1}{8}\tau & \tfrac{1}{8}\tau & -1 \\
19 & \tfrac{1}{2}-\tfrac{1}{2}\tau & \tfrac{1}{2}-\tfrac{1}{2}\tau & \tfrac{1}{8}\tau & \tfrac{1}{8}\tau & 1 \\
20 & \tfrac{1}{2}-\tfrac{1}{2}\tau & \tfrac{1}{2}-\tfrac{1}{2}\tau & \tfrac{1}{4}\tau & 0 & -1 \\
21 & \tfrac{1}{2}-\tfrac{1}{2}\tau & \tfrac{1}{2}-\tfrac{1}{2}\tau & \tfrac{1}{4}\tau & 0 & 1 \\
22 & \tfrac{1}{2}-\tfrac{1}{4}\tau & \tfrac{1}{2}-\tfrac{1}{4}\tau & 0 & 0 & - \tfrac{3}{4} \\
23 & \tfrac{1}{2}-\tfrac{1}{4}\tau & \tfrac{1}{2}-\tfrac{1}{4}\tau & 0 & 0 & \tfrac{3}{4} \\
24 & \tfrac{1}{2}-\tfrac{1}{4}\tau & \tfrac{1}{2}-\tfrac{1}{4}\tau & \tfrac{1}{4}\tau & \tfrac{1}{4}\tau & - \tfrac{3}{4} \\
25 & \tfrac{1}{2}-\tfrac{1}{8}\tau & \tfrac{1}{2}-\tfrac{1}{8}\tau & \tfrac{1}{2}\tau & \tfrac{1}{2}\tau & - \tfrac{1}{2} \\
26 & \tfrac{1}{2} & \tfrac{1}{6}+\tfrac{1}{4}\tau & \tfrac{1}{6}+\tfrac{1}{4}\tau & \tfrac{1}{6}+\tfrac{1}{4}\tau & - \tfrac{1}{2} \\
27 & \tfrac{1}{2} & \tfrac{1}{4}+\tfrac{1}{8}\tau & \tfrac{1}{4}+\tfrac{1}{8}\tau & \tfrac{1}{2}\tau & - \tfrac{1}{2} \\
28 & \tfrac{1}{2} & \tfrac{1}{2}-\tfrac{1}{4}\tau & \tfrac{1}{2}\tau & \tfrac{1}{2}\tau & - \tfrac{1}{2} \\
29 & \tfrac{1}{2} & \tfrac{1}{2} & 0 & 0 & - \tfrac{1}{4} \\
30 & \tfrac{1}{2} & \tfrac{1}{2} & 0 & 0 & \tfrac{1}{4} \\
31 & \tfrac{1}{2} & \tfrac{1}{2} & \tfrac{1}{4}\tau & 0 & 0 \\
32 & \tfrac{1}{2} & \tfrac{1}{2} & \tfrac{1}{4}\tau & \tfrac{1}{4}\tau & - \tfrac{1}{4} \\
33 & \tfrac{1}{2}+\tfrac{1}{4}\tau & \tfrac{1}{2}-\tfrac{1}{4}\tau & 0 & 0 & - \tfrac{1}{4} \\
34 & \tfrac{1}{2}+\tfrac{1}{4}\tau & \tfrac{1}{2}-\tfrac{1}{4}\tau & 0 & 0 & \tfrac{1}{4} \\
35 & \tfrac{1}{2}+\tfrac{1}{4}\tau & \tfrac{1}{2}-\tfrac{1}{4}\tau & \tfrac{1}{4}\tau & \tfrac{1}{4}\tau & - \tfrac{1}{4} \\
36 & \tfrac{1}{2}+\tfrac{1}{2}\tau & 0 & 0 & 0 & 0 \\
37 & \tfrac{1}{2}+\tfrac{1}{2}\tau & \tfrac{1}{6}-\tfrac{1}{12}\tau & \tfrac{1}{6}-\tfrac{1}{12}\tau & \tfrac{1}{6}-\tfrac{1}{12}\tau & 0 \\
38 & \tfrac{1}{2}+\tfrac{1}{2}\tau & \tfrac{1}{4}-\tfrac{1}{8}\tau & \tfrac{1}{4}-\tfrac{1}{8}\tau & 0 & 0 \\
39 & \tfrac{1}{2}+\tfrac{1}{2}\tau & \tfrac{1}{2}-\tfrac{1}{2}\tau & 0 & 0 & 0 \\
40 & \tfrac{1}{2}+\tfrac{1}{2}\tau & \tfrac{1}{2}-\tfrac{1}{2}\tau & \tfrac{1}{8}\tau & \tfrac{1}{8}\tau & 0 \\
41 & \tfrac{1}{2}+\tfrac{1}{2}\tau & \tfrac{1}{2}-\tfrac{1}{2}\tau & \tfrac{1}{4}\tau & 0 & 0 \\
\end{longtable}

\end{document}